\documentclass[reqno]{amsart}

\usepackage{amsmath,amssymb,amsthm}
\usepackage{graphicx,tikz}
\newtheorem{theorem}{Theorem}
\newtheorem*{thm}{Theorem}

\newtheorem{lemma}{Lemma}

\theoremstyle{definition}

\theoremstyle{remark}

\begin{document}

\title[]{Superpolynomial Convergence in the\\ Riemann Rearrangement Theorem}

\author[]{Stefan Steinerberger}
\address{Department of Mathematics, University of Washington, Seattle, WA 98195, USA}
\email{steinerb@uw.edu}

\begin{abstract} 
Let $x \in \mathbb{R}$ be arbitrary and consider the `greedy' approximation of $x$ by signed harmonic sums: given
$a_n = \sum_{k \leq n} \varepsilon_k/k$ with $\varepsilon_k \in \left\{-1,1\right\}$, we set $\varepsilon_{n+1} = 1$ if $a_n \leq x$
and $\varepsilon_{n+1} = -1$ otherwise. Bettin-Molteni-Sanna showed (Adv. Math. 2020) that this procedure
has remarkable approximation properties: for almost all $x \in \mathbb{R}$ one has superpolynomial convergence in the sense that for every $k \in \mathbb{N}$ there are infinitely many $n \in \mathbb{N}$ with $|a_n - x| \leq n^{-k}$.  We extend this result from $\pm 1 \pm 1/2 \pm  1/3 \dots \pm 1/n$ to moment sequences, i.e. sequences defined as the moments of a measure $\mu$ supported on $[0,1]$.  
\end{abstract}

\maketitle

\section{Introduction}
\subsection{A Calculus Problem.}
The origin of this story could be traced to an exercise from Patrick Fritzpatrick's 2009 \textit{Advanced Calculus} \cite{fitz}.
\begin{quote}
\textit{\textbf{Example 2.2}} (Fitzpatrick \cite{fitz})\textbf{.} \textit{ Define $a_1 = 1$. If $n$ is an index such that $a_n$ has been defined, then define
$$ a_{n+1} = \begin{cases} a_n + 1/n \qquad &\mbox{if}~a_n^2 \leq 2 \\
a_n - 1/n \qquad &\mbox{if}~a_n^2 > 2.
\end{cases}$$
This formula defines the sequence recursively. We leave it as an exercise for the reader to find the first four terms of the sequence.}
\end{quote}

This sequence trivially converges to $\sqrt{2}$ and one could ask at what rate. (This question is not asked in \cite{fitz}). 
Since the gap size is $|a_{n+1} - a_n| = 1/(n+1)$, the rate cannot be much better than $\sim 1/n$. On the other hand, there are examples like
 $|a_{98} - \sqrt{2}| \sim 3 \cdot 10^{-8}$, $|a_{850} - \sqrt{2}| \sim 1 \cdot 10^{-15}$ and $|a_{3858} - \sqrt{2}| \sim 9 \cdot 10^{-21}$ which suggest that something interesting is happening.

\subsection{The Bettin-Molteni-Sanna Theorem} A complete explanation was given by Bettin-Molteni-Sanna \cite{bettin}. 
Let $x \in \mathbb{R}$ be arbitrary, $a_0 = 0$, and
$$ a_{n} = \begin{cases} a_{n-1} + 1/n \qquad &\mbox{if}~a_n \leq x \\
a_{n-1} - 1/n \qquad &\mbox{if}~a_n > x.
\end{cases}$$

\begin{thm}[Bettin-Molteni-Sanna, simplified] There is a (countable) set $X \subset \mathbb{R}$
such that if $x \in \mathbb{R} \setminus X$, then for all $k \in \mathbb{N}$ there exist infinitely many $n \in \mathbb{N}$ with
$$ |x - a_n| \leq \frac{1}{n^{k+1}}.$$
\end{thm}

The exceptional set $X$ is needed and not empty.  Consider the example $x = \log{(2)}$. Then, for essentially the same reasons that make the Leibniz alternating series test work, we have $a_1 =1, a_2=  1 - \tfrac12, a_2 = 1 - \tfrac12 + \tfrac13$ and so on: the signs will alternate forever. In particular, for $n$ odd, 
$$ \log{(2)} = \underbrace{ 1 - \frac{1}{2} + \frac{1}{3} - \dots \pm \dots + \frac{1}{n} }_{a_n} - \frac{1}{n+1} + \frac{1}{n+2} - \dots$$
which implies that
$$ \log{(2)} - a_n = \sum_{\ell=1}^{\infty} \frac{(-1)^{\ell}}{n+\ell}  = - \frac{1+o(1))}{2n}.$$
The same argument can be carried out when $n$ is even, the greedy construction cannot approximate $\log{(2)}$ faster than at the trivial rate $1/n^{}$. Bettin-Molteni-Sanna \cite{bettin} prove a number of more precise and additional results. In particular, they also determine all possible limit points of $((x - a_n)\cdot n^k)$ and quantify the dependence of $k$ on $n$: for almost all $x \in \mathbb{R}$ one has
$$ \liminf_{n \rightarrow \infty} ~\frac{ \log |x - a_n|}{(\log n)^2} = - \frac{1}{\log 4}.$$
Thus, for almost all $x \in \mathbb{R}$ there are infinitely many $n \in \mathbb{N}$ such that $|x - a_n| \leq n^{ -0.7 \log{n}}$ which recovers an earlier result \cite{bettin0} obtained via different techniques.

\subsection{Heuristics.}
 Suppose $a_n < x < a_{n+1}$. Then 
 $$ a_{n+2} = a_{n+1} - \frac{1}{n+2} = a_n + \frac{1}{n+1} - \frac{1}{n+2} = a_n + \frac{1}{(n+1)(n+2)}.$$
 This allows for an approximation at the finer scale $n^{-2}$. Then, however, if things go well, we would maybe hope that this repeats itself and that, provided a number of conditions are satisfied, we might end up with
 $$ a_{n+4} =  a_n + \frac{1}{(n+1)(n+2)} -  \frac{1}{(n+3)(n+4)} \sim a_n  + \frac{4}{n^3}$$
allowing for an even finer approximation. To a first approximation, this is what happens. However, things are not so simple (as also evidenced by the existence of the exceptional set $X$). Another indication that things are not entirely trivial is the appearance of the Thue-Morse sequence. The Thue-Morse sequence was discovered by Thue \cite{thue, thue2} in his study of combinatorics on words
and, independently, by Morse \cite{morse} when considering of geodesics on surfaces \cite{hedlund}. One starts with 0 and then iteratively appends the Boolean complement
$$ 0 \rightarrow 0\textbf{1} \rightarrow 01 \textbf{10} \rightarrow 0110 \textbf{1001} \rightarrow 01101001\textbf{10010110} \rightarrow \dots$$
The Thue-Morse sequence is the limit
 $0110100110010110....$
This sequence shows up here: consider, for example, $x = 0.8$.  After the 55th element $a_{55}$, the next 12 sign changes correspond exactly to the first 12 elements of the Thue-Morse sequence
\begin{align*}
 a_{67} = a_{55} \underset{\textbf{0}}{-} \frac{1}{56}  \underset{\textbf{1}}{+} \frac{1}{57}  \underset{\textbf{1}}{+} \frac{1}{58}  \underset{\textbf{0}}{-} \frac{1}{59}  \underset{\textbf{1}}{+} \frac{1}{60}  \underset{\textbf{0}}{-} \frac{1}{61} \underset{\textbf{0}}{-} \frac{1}{62}  \underset{\textbf{1}}{+} \frac{1}{63}  \underset{\textbf{1}}{+} \frac{1}{64}  \underset{\textbf{0}}{-} \frac{1}{65}  \underset{\textbf{0}}{-} \frac{1}{66}  \underset{\textbf{1}}{+} \frac{1}{67}.
 \end{align*}
 This is not a coincidence and something that any proof has to account for. The argument given by Bettin-Molteni-Sanna \cite{bettin} is quite nontrivial.

\section{Result}
\subsection{General sequences.} We were motivated by the question of what happens if we replace $(1/n)$ by a more general sequence $(x_n)$. Suppose $x \in \mathbb{R}$ is a number to be approximated and $(x_n)_{n=1}^{\infty}$ is a sequence of reals.  We define the sequence $$ a_0 = 0 \qquad \mbox{and} \qquad  a_{n} = \begin{cases} a_{n-1} + x_n \qquad &\mbox{if}~a_{n-1} \leq x \\
a_{n-1} - x_n \qquad &\mbox{if}~a_n > x. \end{cases}$$
It is clear that, at that level of generality, one may not be able to hope for much. There is a trivial obstruction: if the sequence $(x_n)$ decays too quickly, one cannot hope to approximate much of anything. The other obstruction is that superpolynomial convergence requires $(x_n)$ to have some type of `smoothness' property that allows for long-range cancellations. An example is given by $x_k  = 1/p_k$, where $p_k$ is the $k-$th prime number: we see that the error is of order $\sim 1/p_k$ half the time and appears to sometimes be $\sim 1/p_k^2$ but never much better than that.

\begin{center}
\begin{figure}[h!]
\begin{tikzpicture}
\node at (0,0) {\includegraphics[width=0.4\textwidth]{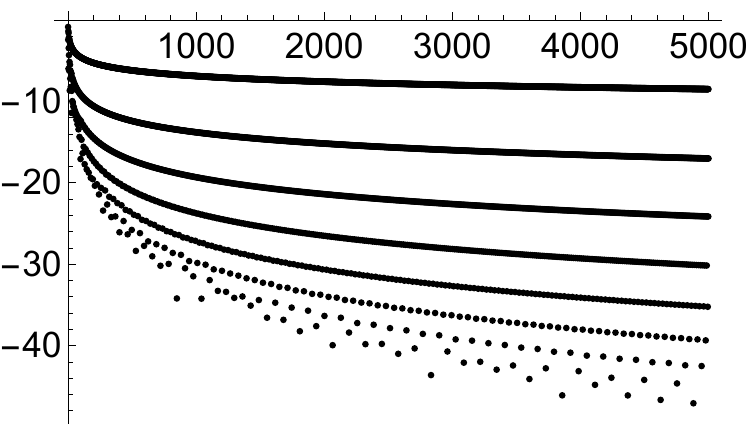}};
\node at (6,0) {\includegraphics[width=0.4\textwidth]{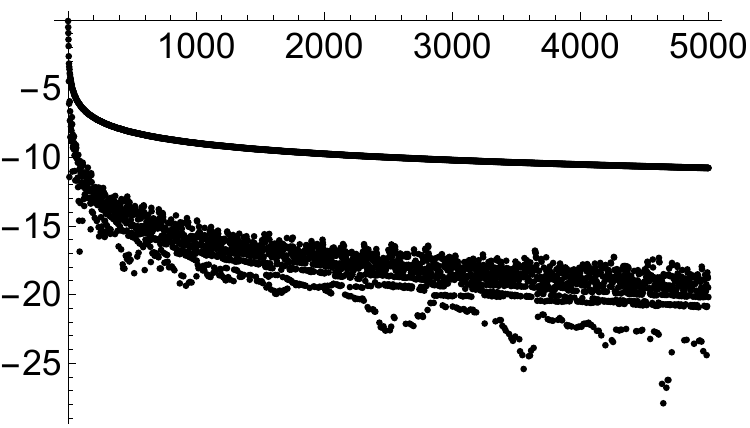}};
\end{tikzpicture}
\caption{The approximation error $\log{|a_n - \sqrt{2}|}$ for two different sequences: the usual $x_k = 1/k$ (left) and
$x_k = 1/p_k$ where $p_k$ is the $k-$th prime number (right).}
\end{figure}
\end{center}

\subsection{Moment sequences.} 
We show that \textit{moment sequences} as defined by Hausdorff \cite{hausdorff, hausdorff2, hausdorff3} have the superpolynomial approximation property. 
\begin{quote} \textbf{Definition.} For any Borel measure $\mu$ compactly supported on the unit interval $[0,1]$, we define the \textit{moment} sequence $(x_n)_{n=0}^{\infty}$ via
$$ x_n = \int_0^1 t^n d\mu.$$
\end{quote}

 If $\mu$ is the Lebesgue measure, then we recover the harmonic summands since
$$x_n = \int_0^1 t^n dt = \frac{1}{n+1}.$$
The asymptotic behavior of $x_n$ is determined by the behavior of the measure close to 1. If the measure behaves like $\mu \sim dx$, we recover $x_n \sim 1/n$. A slower rate can be attained by having the measure blow up (in an integrable way) close to 1. For example, considering the measure $\mu = dx/\sqrt{\pi - \pi x}$, we end up with 
$$ x_n = \int_0^1 t^n  \frac{1}{\sqrt{\pi}}\frac{ dt }{\sqrt{1 -  t}}=\frac{\Gamma(n+1)}{\Gamma\left(n+ \frac{3}{2} \right)} = \frac{1}{\sqrt{n}} - \frac{3}{8} \frac{1}{n^{3/2}} + \dots$$
Another nice example that also decays like $\sim 1/\sqrt{n}$ is given by setting $\mu$ to be a shifted and rescaled Wigner distribution
$$ \int_0^1 t^n   \frac{1}{\pi} \sqrt{\frac{1}{4} - \left(t - \frac{1}{2} \right)^2}  \frac{dt}{1-t}  = \frac{1}{2^{2n+1}} \binom{2n+1}{n+1} = \frac{1}{\sqrt{\pi n}} - \frac{5}{8} \frac{1}{\sqrt{\pi}} \frac{1}{n^{3/2}} + \dots $$
One could consider more exotic examples: if $\mu$ is the canonical probability measure on the fractal middle third Cantor set, then the moments are given by the sequence $1, 1/2, 3/8, 5/16, 87/320, 31/128, \dots$, this sequence decays asymptotically like $\sim n^{- \log{2}/\log{3}}$ (see \cite{goh}). A faster rate of decay can be achieved by having the measure $\mu$ vanish close to 1, for example
$$  \int_0^1 t^n \log{\left( \frac{1}{t} \right)}  \frac{dt}{t} = \frac{1}{n^2}.$$

 \subsection{Slow decay.} We can now state the main results. There is a clear difference between moment sequences $(x_n)$ that decay so slowly that $\sum x_n$ diverges and moment sequences that decay faster. If the series diverges, then we get superpolynomial approximation for all but a few exceptional real numbers.
\begin{theorem}
Let $(x_n)_{n=1}^{\infty}$ be a moment sequence with, asymptotically, slow polynomial decay, i.e. there exist $c > 0$ and $0 < \alpha \leq 1$ such that
$$  \lim_{n \rightarrow \infty} x_n \cdot n^{\alpha} = c.$$
There exists a countable set $X \subset \mathbb{R}$ so that for all $x \in \mathbb{R} \setminus X$  
$$ a_0 = 0 \qquad \mbox{and} \qquad  a_{n} = \begin{cases} a_{n-1} + x_n \qquad &\mbox{if}~a_{n-1} \leq x \\
a_{n-1} - x_n \qquad &\mbox{if}~a_n > x\end{cases}$$
has the superpolynomial approximation property: for all $k \in \mathbb{N}$ there are infinitely many $n \in \mathbb{N}$ where
$$ |x - a_n| \leq \frac{1}{n^k}.$$
\end{theorem}

The conditions in this result are essentially sharp: divergence of $\sum x_n$ is clearly required if one wants to be able to approximate arbitrarily large real numbers. The necessity of an exceptional set $X \subset \mathbb{R}$ was already illustrated above. The proof shows a few other things as a byproduct, these are discussed in \S 4.

 \subsection{Fast decay.} We will now consider moment sequences that decay faster than $1/n$. Consider, the example $x_n = 1/n^2$ from above. The series converges, it will be impossible to approximate, for example, any $x \geq \pi^2/6$. However, it may be able to approximate numbers smaller than that. Trying to approximate $x = \sqrt{2}$, we see $|\sqrt{2} - a_{4566}| \sim 1 \cdot 10^{-24}$, there is reason for hope.
As mentioned in \S 2.1, there are two types of obstructions for general sequences. Divergence is clearly necessary if one wants to approximate all but an exceptional set of real numbers: this is the first trivial obstruction. The second obstruction is the presence of some type of long-range correlation that can lead to cancellation over many orders of magnitude. Moment sequences always have this required long-range correlation, the only obstruction is whether one can reach the number to be approximated. This leads to a slightly more refined result that essentially says that  if one can get close within a finite number of steps and if the $(x_n)$ do no vary greatly in scale, then one has the superpolynomial approximation property.

\begin{theorem}
Let $(x_n)_{n=1}^{\infty}$ be a moment sequence with asymptotically polynomial decay, i.e. there exist $c, \alpha > 0$ such that
$$  \lim_{n \rightarrow \infty} x_n \cdot n^{\alpha} = c.$$

There exists a countable set $X \subset \mathbb{R}$ so that for all $x \in \mathbb{R} \setminus X$, the sequence
$$ a_0 = 0 \qquad \mbox{and} \qquad  a_{n} = \begin{cases} a_{n-1} + x_n \qquad &\mbox{if}~a_{n-1} \leq x \\
a_{n-1} - x_n \qquad &\mbox{if}~a_n > x\end{cases}$$
has the following property: if there exists an index $a_m$ such that
$$ \min\left\{a_m, a_{m+1} \right\} \leq x \leq  \max\left\{a_m, a_{m+1} \right\}$$
and if there exists a finite $\ell \in \mathbb{N}$ such that for all $j \geq m$
$$ x_j \leq x_{j+1} + x_{j+2} + \dots + x_{j+ \ell},$$
then the sequence $(a_n)$ superpolynomially approximates $x$ and for all $k \in \mathbb{N}$ there are infinitely many $n \in \mathbb{N}$ where
$$ |x - a_n| \leq \frac{1}{n^k}.$$
\end{theorem}
 The condition
$ \min\left\{a_m, a_{m+1} \right\} \leq x \leq  \max\left\{a_m, a_{m+1} \right\}$
ensures that the sequence eventually gets close to $x$ which, given the faster decay of the sequence, may no longer be the case. 
The second condition, the existence of a uniform $\ell \in \mathbb{N}$ such that, after the crossing of $x$, we have
$ x_j \leq x_{j+1} + x_{j+2} + \dots + x_{j+ \ell}$
is automatically satisfied for large enough $j$ since $x_j = (c + o(1))/j^{\alpha}$. One can choose $\ell=2$ for all indices large enough. The condition is needed: consider approximation the number $x = 0.2$ using $x_n = 1/n^2$. Then $0 = a_0 \leq x \leq a_1 = 1$
but, for all $\ell \geq 1$, we have $x_2 + x_3 + \dots + x_{\ell} < 0.7 < 1 =x_1.$
This is unsurprising: even though the sequence manages to cross $x$, there is not enough left in the remainder to return to $x$.
If $x_n = (c+o(1))/n^{\alpha}$ for some $0 < \alpha \leq 1$, then all the extra conditions in Theorem 2 are automatically satisfied and Theorem 2 implies Theorem 1.

\section{Proofs}
 Our proof is an induction on scales argument that bypasses the Thue-Morse sequence. We first derive some basic properties of moment sequences. 
\S 3.2 contains the Main Lemma which does most of the heavy lifting. The proof of Theorem 1, in \S 3.3, ensures that Lemma 2 can be applied in an iterative fashion. The proof of Theorem 2 follows mostly the same lines and is described in \S 3.4.

\subsection{Moment sequences.} We quickly recall some elementary facts.
\begin{lemma} Suppose $(x_n)_{n=1}^{\infty}$ is a moment sequence and
$$ x_n = (1+o(1)) \frac{c}{n^{\alpha}} \qquad \mbox{for some}~c, \alpha > 0.$$
$x_n$ is monotonically decreasing. The sequence of shifts $(x_{n} - x_{n+1})_{n=1}^{\infty}$ is also a moment sequence and satisfies
$$ x_n - x_{n+1} = (1+o(1)) \frac{c \cdot \alpha }{n^{\alpha + 1}}.$$
\end{lemma}
\begin{proof} By definition, $ x_n = \int_0^1 t^n d\mu \geq 0$ for some $\mu$ supported on $[0,1]$.
Then
$$ x_n - x_{n+1} = \int_0^1 (t^n - t^{n+1}) d\mu =  \int_0^1 t^n (1 - t) d\mu \geq 0.$$
This shows that $x_n - x_{n+1} \geq 0$ and it also shows that the sequence of differences is a moment sequence generated by the new measure $(1-t) d\mu$. 
Since $(x_n - x_{n+1})$ is also a moment sequence, it, too, is monotonically decreasing. This implies that, for any $\varepsilon >0$ and $n$ sufficiently large (depending on $\varepsilon$),
$$ \frac{x_{(1-\varepsilon) n} - x_{n}}{ \varepsilon n}  \geq  \frac{x_n - x_{n+1}}{1}  \geq \frac{x_n - x_{(1+\varepsilon) n}}{ \varepsilon n}.$$
The left-hand side can be, asymptotically, written as
$$ \frac{1}{\varepsilon n} \left( \frac{c +o(1)}{(1-\varepsilon)^{\alpha} n^{\alpha}} - \frac{c + o(1)}{n^{\alpha}} \right) = \frac{c +o(1)}{n^{\alpha+1}}  \frac{1}{\varepsilon} \left( \frac{1}{(1-\varepsilon)^{\alpha}} - 1\right).$$
Taking the limit $\varepsilon \rightarrow 0^+$ shows that the factor is $\alpha$. The argument for the right-hand side is identical and this establishes the desired result.
\end{proof}

\subsection{The Main Lemma.} Since our argument is recursive, the Main Lemma requires a very precise formulation that allows for bootstrapping.

\begin{lemma}[Main Lemma] Suppose $(x_n)_{n=1}^{\infty}$ is a moment sequence satisfying
$ x_n = (c+o(1))/n^\alpha$ for some $c, \alpha> 0$.
Suppose that $a_n \leq x \leq a_{n+1}$. There exists a countable set $X \subset \mathbb{R}$ such that for all $x \in \mathbb{R} \setminus X$ and for infinitely many $m \geq n$
$$ |x - a_{m}| \leq x_{m+1} - x_{m+2}.$$
\end{lemma}

\begin{proof}
We only consider the case $a_n \leq x \leq a_{n+1}$. The case $a_n \geq x \geq a_{n+1}$ is completely symmetric and follows from making all the necessary sign changes.\\
\textbf{1.} We first show that once the sequence $(a_n)$ is `fairly close' to $x$ in the sense of $a_n \leq x \leq a_{n+1}$, it is going to stay fairly close for all subsequent times: for all $m \geq n+1$, we have $|x - a_m| \leq x_m$. This is a simple proof by induction. Suppose that $ a_n \leq x \leq a_{n+1} = a_n + x_{n+1}$. This implies $|a_{n+1} - x| \leq x_{n+1}$. Suppose now that $|a_{m} - x| \leq x_{m}$. If $a_m > x$,
then $a_{m+1} = a_m - x_{m+1}$ and  $ \left| a_{m+1} - x \right|  \leq |x_{m+1}|.$
The case $a_m \leq x$ is identical and the desired statement follows.

\begin{center}
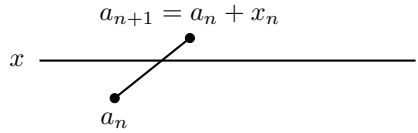
\begin{figure}[h!]
\begin{tikzpicture}
\draw [thick] (0,0) -- (5,0);
\node at (-0.3, 0) {$x$};
\filldraw (1,-0.5) circle (0.06cm);
\node at (1, -0.8) {$a_n$};
\filldraw (2,0.3) circle (0.06cm);
\node at (2, 0.6) {$a_{n+1} = a_n + x_n$};
\draw [thick] (1, -0.5) -- (2, 0.3);
\end{tikzpicture}
\vspace{-10pt}
\caption{The setting for Lemma 2.}
\end{figure}
\end{center}
\vspace{-10pt}

\textbf{2.} If, starting at a certain arbitrary index $m \in \mathbb{N}$, all subsequent choices of sign are always alternating, meaning that
$$ \forall~\ell \in \mathbb{N} \qquad \qquad a_{m+\ell} = a_m  + x_{m+1} - x_{m+2} + x_{m+3} - \dots - (-1)^{\ell} x_{m+\ell}  $$
or 
$$ \forall~\ell \in \mathbb{N} \qquad \qquad  a_{m+\ell} = a_m  - x_{m+1} + x_{m+2} - x_{m+3} + \dots - (-1)^{\ell}  \pm x_{m+\ell} ,$$
then this uniquely identifies the real number $x$. The two cases are identical (considering the second case for $m+1$ recovers the first case with shifted indices).  We may thus assume that we are in the first case. The fact that the first choice of sign is $+$ implies $a_m \leq x$. The second choice being $-$ implies $a_m + x_{m+1} > x$. The third choice being $+$ again implies $a_m + x_{m+1} - x_{m+2} \leq x$. We deduce, for all $\ell \in \mathbb{N}$,
$$ a_m + x_{m+1} -   \dots \pm \dots - x_{m+2\ell} \leq x < a_m + x_{m+1} -   \dots \pm \dots - x_{m+2\ell} + x_{m + 2\ell + 1}.$$
Abbreviating the left-hand side as
$ y_{\ell} = a_m + x_{m+1} -  \dots - x_{m+2\ell},$
we note that, $(x_n)$ being a moment sequence and thus decreasing, $ y_{\ell + 1} = y_{\ell} + x_{m+2\ell + 1} - x_{m + 2\ell + 2}  \geq y_{\ell}.$
Likewise, defining the upper bound to be $z_{\ell}$, we see that $z_{\ell}$ is decreasing. Note that $y_{\ell} \leq x < z_{\ell}$ and that both $y_{\ell}$ and $z_{\ell}$ are monotone and bounded: both have a limit. Since $z_{\ell} - y_{\ell} = x_{m+2\ell + 1} \rightarrow 0$ tends to 0, the limit is $x$ and 
$$ x = a_m + x_{m+1} - x_{m+2} + x_{m+3} - x_{m+4} + x_{m+5} - \dots  $$
\textbf{3.} The previous argument shows a little bit more: suppose we have a real number $x$ such that its sign pattern contains only a finite number of consecutive $(+, +)$ or $(-, -)$.  Then, after a certain finite index $N \in \mathbb{N}$, the sign pattern is alternating.  This identifies a unique real number.  Since the number $a_{N}$ can be written as $a_{N} = \pm x_1 \pm x_2 \pm \dots \pm x_N$ for a suitable choice of signs, we know that $a_N$ can only assume $2^N$ different values. This corresponds to a finite number of possible $x$ that exhibit this behavior. Taking a union over all $N$ and recalling that a countable union of countable sets is countable, we see that the sign patterns $(+, +)$ or $(-, -)$ appear an infinite number of times for all but at most countably many $x$.\\
\textbf{4.} Suppose now that we observe the consecutive sign choice $(+, +)$ an infinite number of times. As usual, the case $(-,-)$ is analogous. Then, for a sufficiently large index $m \geq n$, we have $x_m \sim 1/m^{\alpha}$ and we also have $|x-a_m| \leq x_m$. This means that there can be at most a finite number of consecutive sign choices that are $+$ and we can find a $-$ at the start of each chain of $(+, +)$, and therefore an occurrence of $(-, +, +)$. This sign pattern means that 
$ a_{m+3} = a_m - x_{m+1} +  x_{m+2} + x_{m+3}.$
This choice of sign implies that $a_m > x$ as well as $a_m - x_{m+1} + x_{m+2} \leq x$. Then
 $$ a_m - (x_{m+1} - x_{m+2}) \leq x < a_m$$
 which is the desired result.
\end{proof}

\subsection{Proof of Theorem 1}
Our goal is to prove that there are infinitely many integers $n \in \mathbb{N}$ with the property that
$ |x - a_n| \leq 1/n^k.$
The result has a nice constant 1 that will be obtained by a simple trick: since $k$ is arbitrary, it suffices to prove that for every $k \in \mathbb{N}$ there exist infinitely many $n \in \mathbb{N}$ such that
$ |x - a_n| \leq c_k/n^k.$
for some constant $c_k < \infty$. One can then use this result with $k+1$ instead of $k$ and, for $n$ sufficiently large, one arrives at the nice constant 1.

\begin{proof}  We fix $k$. We
choose $N$ sufficiently large so that for all $n \geq N$, the sequence $x_n$ and its first $k$ `discrete' derivatives are all reasonably well approximated by their respective asymptotic limit. Since  $x_n \cdot n^{\alpha} \rightarrow c$, there exists $N_1 \in \mathbb{N}$ such that for all $n \geq N_1$
$$ \frac{c/2}{n^{\alpha}} \leq x_n \leq \frac{2c}{n^{\alpha}}.$$
Lemma 1 ensures that requiring $x_n \cdot n^{\alpha}$ to have a limit ensures that $(x_{n} - x_{n+1})\cdot n^{\alpha +1}$ has one too and we deduce that for all $n \geq N_2$
 $$ \frac{c \alpha/2}{n^{\alpha+1}} \leq x_n - x_{n+1} \leq \frac{2c \alpha}{n^{\alpha+1}}$$
and, for all $n \geq N_3$
 $$ \frac{c \alpha (\alpha + 1)/2}{n^{\alpha+1}} \leq (x_n - x_{n+1}) - (x_{n+1} - x_{n+2}) \leq \frac{2c \alpha (\alpha + 1)}{n^{\alpha+2}}$$
 and so on for the first $k$ discrete differences. Set $N = \max(N_1, \dots, N_k)$. We additionally want that
 $$ \frac{1}{N^k} \leq \frac{1}{5} \left( \frac{1}{(N+1)^{k+1}} +  \frac{1}{(N+2)^{k+1}} + \dots +  \frac{1}{(N+10)^{k+1}} \right) \qquad \quad (\diamond)$$
 and if this is not satisfied for $N$, then we increase its value further until that inequality is satisfied as well. We now consider the sequence $a_n$ for $n \geq N$.
 Since $\sum_{n = N+1}^{\infty} x_n = \infty$, even when only considering after the first $N$ elements, the sequence is guaranteed to cross the level $x$ at some point, i.e. there exists $m \geq N$ such that
 $$ a_{m} > x \geq a_{m+1} \qquad \mbox{or} \qquad a_m \leq x < a_{m+1}.$$
 This puts us into the situation where Lemma 2 is applicable. Unless $x$ happens to be a very particular real number (from a countable set), we can deduce that for some $m \geq N$ the sequence is very close to $x$ and $|a_m - x| < x_m - x_{m+1}$ (Lemma 2 ensures that there are infinitely many indices where this is true and, in particular, there are also indices that are larger than $N$ where this is true).  Using the quantitative control on the sequence $x_n$ and its differences
$$ |a_m - x| < x_{m+1} - x_{m+2} \leq \frac{2c \alpha}{m^{\alpha + 1}} \leq \frac{4 \alpha}{m}  \frac{c/2 \alpha}{m^{\alpha}} \leq \frac{4\alpha}{m} x_m.$$ 

We have found an approximation that is an entire order of magnitude smaller than the step size $x_m$. If $a_m \leq x$, then the next sign choice is $+$ and we surely have $a_{m+1} > x$, the subsequent sign is $-$. Conversely, if $a_m > x$, then the next two signs are going to be $(-, +)$. This is perhaps the most crucial step in the argument.

\begin{center}
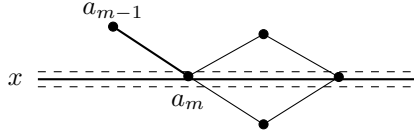
\begin{figure}[h!]
\begin{tikzpicture}
\draw [thick] (0,0) -- (5,0);
\node at (-0.3, 0) {$x$};
\filldraw (1,0.7) circle (0.06cm);
\node at (2, -0.3) {$a_m$};
\filldraw (2,0.04) circle (0.06cm);
\node at (1, 0.9) {$a_{m-1}$};
\draw [thick] (1, 0.7) -- (2, 0.04);
\draw [dashed] (0, 0.1) -- (5, 0.1);
\draw [dashed] (0,- 0.1) -- (5, -0.1);
\filldraw (3,0.6) circle (0.06cm);
\filldraw (4,0.03) circle (0.06cm);
\filldraw (3,-0.6) circle (0.06cm);
\draw (2, 0.04) -- (3, 0.6) -- (4, 0.03);
\draw (2, 0.04) -- (3, -0.6) -- (4, 0.03);
\end{tikzpicture}
\caption{The crucial step: once $a_m$ is unexpectedly close to $x$, the next two steps are fixed to either be $(+, -)$ or $(-, +)$. Moreover, the choice turns out to be consistent with picking $\pm (x_{m+1}  - x_{m+2})$ depending on whether $a_m < x$ or $a_m >x$.}
\end{figure}
\end{center}

 Summarizing, we have
$$ a_{m+2} = a_m + \begin{cases} +(x_{m+1} - x_{m+2}) \qquad &\mbox{if} ~ a_m \leq x \\
-(x_{m+1} - x_{m+2}) \qquad &\mbox{if}~a_m > x. \end{cases}$$
This, however, is \textit{exactly} the type of sequence that we are already studying, except that the moment sequence $(x_n)$ has been replaced by another moment sequence, the moment sequence of gaps $(x_n - x_{n+1})$. This only happened because $a_m$ was unexpectedly close to $x$, however, as we have already seen at the beginning of the proof of Lemma 2, this property stays preserved for sequences of this type. Thus, we will also have
$$ a_{m+4} = a_{m+2} + \begin{cases} +(x_{m+3} - x_{m+4}) \qquad &\mbox{if} ~ a_{m+2} \leq x \\
-(x_{m+3} - x_{m+4}) \qquad &\mbox{if}~a_{m+2} > x \end{cases}$$
and so on for all subsequent terms. We may consider this as passing from $0-$th approximation level ($|a_n - x| \lesssim x_n \lesssim n^{-\alpha}$)
to the first order approximation level ($|a_n - x|  \lesssim n^{-\alpha - 1}$) along a subsequence of density $1/2$. At this point, we would like to bootstrap everything: defining the new moment sequence $y_n = x_{m+n} - x_{m + n+1}$, we would now like to approximate the number $x - a_m$. To show that the argument can be repeated, one more step is required: Lemma 2 requires that we start with two consecutive terms that cross the level $x$. This is easy to do when starting with the $0-$th approximation level since the original moment sequence $(x_n)$ is divergent. This is no longer true when we are at the $1-$th approximation level or any of the other approximation level since the new moment sequences decay at a faster rate and are convergent. Luckily, we are sufficiently close: recall that the distance between $a_m$ and $x$ satisfies
 $$ |a_m - x| \leq  x_{m+1} - x_{m+2} \leq \frac{2c \alpha}{m^{\alpha + 1}}.$$
The elements of the gap sequence satisfy a corresponding lower bound at the same scale and, for all $k \geq m$, 
 $$ x_k - x_{k-1} \geq \frac{c \alpha/2}{k^{\alpha + 1}}.$$
 Since $(\diamond)$ is satisfied, this ensures a crossing of the new subsequence within at most 10 steps. We can then apply Lemma 2 to the subsequence and obtain the desired result. The argument can be continued indefinitely as long as all the discrete derivatives satisfy sufficiently strong uniform estimates (which we can ensure a priori for any finite number of derivatives as long as we are willing to drop an initial sequence of terms). We conclude with a quick note on the exceptional set: by Lemma 2, the sequence $(x_n)_{n=1}^{\infty}$ leads to a subsequence at a lower level for all but a countable exceptional set of real numbers. The same is true for the gapped sequence $(x_n - x_{n+1})$ and the 2-gapped sequence $((x_n - x_{n+1}) - (x_{n+1} - x_{n+2})) = (x_n - 2x_{n+1} + x_{n+2})$ and so on; the countable union of countable sets remains countable.
\end{proof}

\subsection{Proof of Theorem 2}

The proof of Theorem 2 is essentially the proof of Theorem 1 coupled with a careful analysis where the divergence of $\sum x_n$ was used in the proof of Theorem 1.  The moment we pass from the trivial $0-$th approximation level ($|a_n - x| \lesssim x_n \lesssim n^{-\alpha}$)
to the first order approximation level ($|a_n - x|  \lesssim n^{-\alpha - 1}$) along a subsequence, we lose the divergence of the sum over the attached moment sequence since $x_n - x_{n+1} \lesssim x_n /n$.  Indeed, the only time the divergence $\sum x_n$ is used at the very beginning of the argument: we want to enforce that both the sequence $(x_n)$ and the first $k$ `discrete derivatives' are all simultaneously within a factor of 2 of their asymptotic limit. Since all of them converge, this can always be ensured by restricting indices to some $n \geq N$ with $N$ sufficiently large. Then, since $\sum x_n$ diverges, we know that  $\sum_{n \geq N} x_n$ diverges as well, we can find some $a_m \leq x \leq a_{m+1}$ with $m \geq N$ and then use Lemma 2 and, employing uniform control on the derivative, keep applying Lemma 2 iteratively, $k$ more times because that's how many levels of derivatives we have uniform control, to get the desired result. We will now show that this is still possible with the new assumptions.\\
\textit{Claim.} There exist infinitely many $n$ so that $ a_n \leq x \leq  a_{n+1}.$

\begin{proof}[Proof of Claim] By assumption there exists at least one $m \in \mathbb{N}$ such that
$$ \min\left\{a_m, a_{m+1} \right\} \leq x \leq  \max\left\{a_m, a_{m+1} \right\}.$$
We have $|a_{m+1} - x| \leq x_{m+1}$. Since
$$ x_{m+1} \leq x_{m+2} + x_{m+3} + \dots + x_{m+\ell+1},$$
there is going to be another sign change within the next $\ell$ steps and we can iteratively show that a sign change
happens at least once every $\ell$ steps. \end{proof}
This allows us, assuming the setting of Theorem 2, to restrict ourselves to $n \geq N$ for any $N$ arbitrarily large and still ensure that Lemma 2 will be applicable. Then, as in the proof of Theorem 1, we choose $N$ so large that $(x_n)$ and the first $k$ discrete derivatives are sufficiently close to their asymptotic limit and to ensure that  $(\diamond)$ which implies that the gapped subsequence can cross $x$ again (at least $k$ times). The proof is then identical to the proof of Theorem 1.  

\section{Comments and Remarks}
\subsection{Densities.}  The argument shows a little bit more: for example, in the case of slow decay $x_n = (c+ o(1))/n^{\alpha}$ with $0 < \alpha \leq 1$, the argument immediately implies that
the set of integers for which $|a_n - x| \sim 1/n$ is, for almost all $x \in \mathbb{R}$, going to have density 1/2.
Iterating the argument on the next sub-level, we see $|a_n - x| \sim 1/n^k$ is going to happen for a set
of integers with density $2^{-k}$. 

\subsection{Convergence rates.} Incorporating more information about the error terms, one could obtain more precise convergence
rates. We quickly sketch what this would look like. Suppose $a_n \sim c n^{-\alpha}$. Provided we have reached $a_n \leq x \leq a_{n+1}$, we have $|a_{n} - x| \lesssim c n^{-\alpha}/2$ or $|a_{n+1} - x| \lesssim c n^{-\alpha}/2$. The moment we have the same sign choice in a row, we jump to the next approximation level. Until then, we have alternating sign choices at rate $x_{n+1} - x_{n} \sim c \alpha n^{-\alpha - 1}$.  This means that we might need to wait $2X$ iterations where $X$ is chosen so that
$$ \frac{c}{2 n^{\alpha}} \sim \sum_{k = n}^{X} \frac{c \alpha}{n^{\alpha + 1}} \sim c \int_n^X \frac{\alpha}{z^{\alpha+1}} dz = c \left( \frac{1}{n^{\alpha}} - \frac{1}{X^{\alpha}}\right).$$
This means, we have to wait, in the worst case, $2X \sim 2^{1 + 1/\alpha} n$ steps until hitting the next level. Rephrasing this, we expect to be able to get level $\sim \log{n}$ within the first $n$ steps. In the case of the harmonic series, this mirrors the very precise result obtained by Bettin-Molteni-Sanna \cite{bettin}.

\subsection{Special cases.} It is clear that for special cases much more can be said. Considering, for example, $x_n = 1/n^2$, the condition
$$ x_j \leq x_{j+1} + x_{j+2} + \dots + x_{j+ \ell} \qquad \qquad \qquad (\diamond)$$
fails for $j=1$ but is true for all $j \geq 2$ when $\ell = 5$.  This means that, in a suitable sense, the first term is the only problem, the only term that cannot be compensated for later. Indeed, it seems that every real number in $(2 - \pi^2/6, \pi^2/6)$ can be superpolynomially approximated using $\pm 1/n^2$. A similar type of analysis is possible for all polynomially decaying sequences where $(\diamond)$ is always eventually satisfied, the complications only come from a finite number of terms.

\subsection{Higher dimensions.}
The surprising convergence phenomenon reported by Bettin-Molteni-Sanna \cite{bettin} suggests the question whether there might be interesting analogues in higher dimensions. Perhaps the simplest special case was considered in \cite{stein}. Suppose $x_0 = \mathbf{0} \in \mathbb{R}^n$ and $(v_n)_{n=1}^{\infty}$ is a sequence of vectors in $\mathbb{R}^n$, consider the choice of signs that moves one closer to the origin.
$$ x_{n} = \begin{cases} x_{n-1} + v_n \qquad &\mbox{if}\quad \|x_{n-1}+ v_n\| < \|x_{n-1} - v_n\| \\ x_{n-1} - v_n \qquad &\mbox{if}\quad \|x_{n-1}- v_n\| < \|x_{n-1} + v_n\|. \end{cases}$$

\begin{center}
\begin{figure}[h!]
\begin{tikzpicture}
\node at (0,0) {\includegraphics[width=0.25\textwidth]{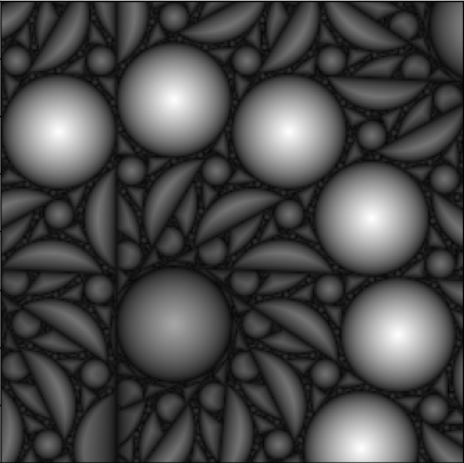}};
\node at (6,0) {\includegraphics[width=0.25\textwidth]{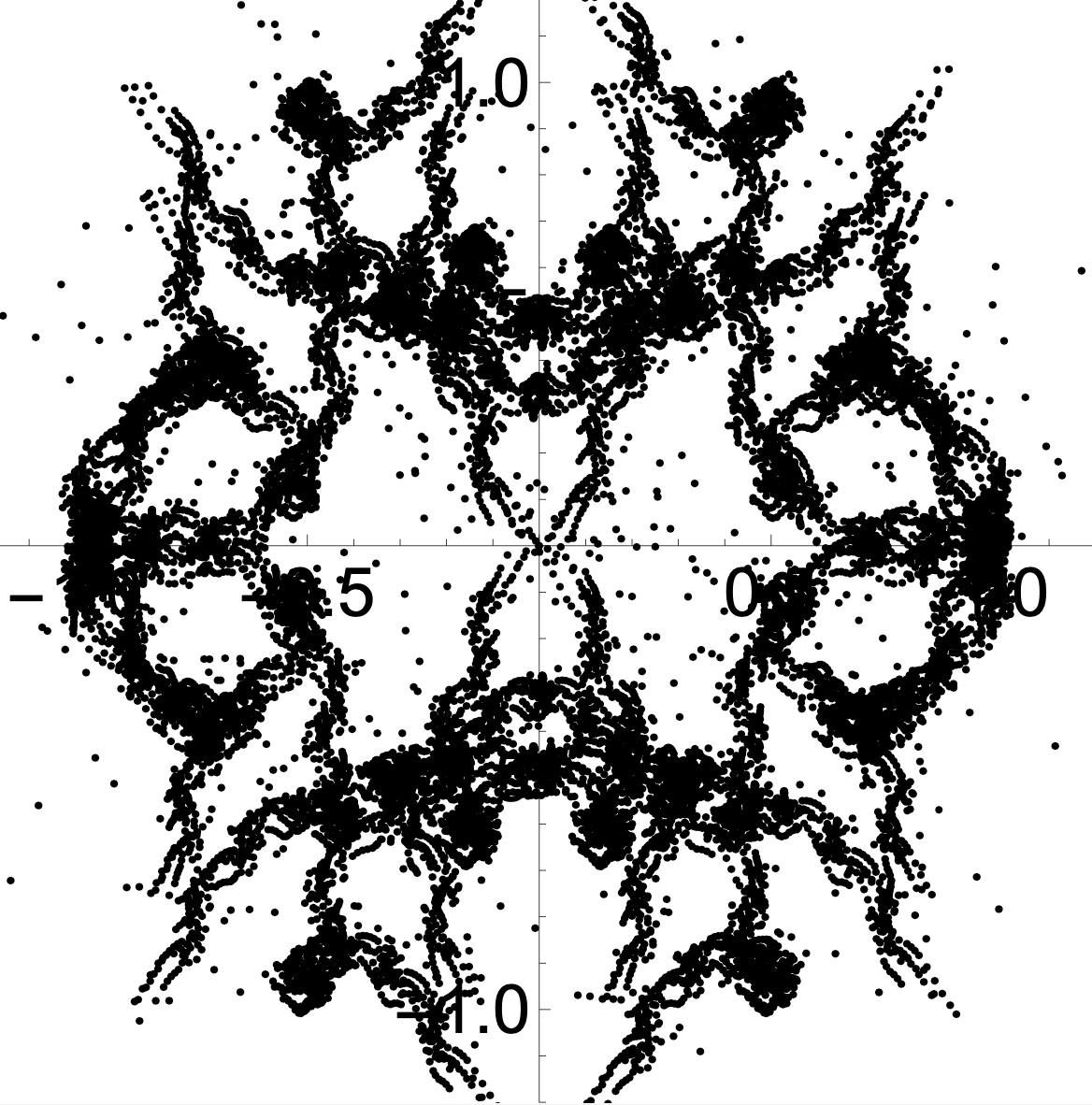}};
 
\end{tikzpicture}
\caption{Left: a figure from \cite{stein}, right: the `Rorschach' sequence arising from $v_n = \exp(2 \pi i \|\sqrt{3} n\|) $ reported in \cite{albors}.}
\end{figure}
\end{center}

This was considered by Steinerberger-Zeng \cite{stein} in $\mathbb{R}^2 \cong \mathbb{C}$ for the special case of $v_n = e^{2\pi i n \alpha}$, the arising sequence can be at least partially explained. Another surprising example \cite{albors} is the following: using $\| \cdot \|$ to denote the distance to the nearest integer, i.e.
$ \|x\|= \min( x- \left\lfloor x \right\rfloor, \left\lceil x \right\rceil - x)$, the sequence of vectors $v_n = \exp(2 \pi i \|\sqrt{3} n\|) $
leads to the `Rorschach' picture which remains unexplained.

\end{document}